\documentclass[a4paper]{amsart}
\usepackage{enumitem}
\usepackage[T1]{fontenc}
\usepackage{amsmath, amsthm, bbm, stackengine, thmtools, amssymb}
\usepackage{mathtools}
\usepackage{hyperref, cleveref, url}
\usepackage[utf8]{inputenc}
\usepackage{tikz-cd}
\usepackage{xspace}
\usepackage{cite}

\def\temp{&} \catcode`&=\active \let&=\temp

\usepackage{rotating}
\tikzset{cong/.style={above,edge node={node [sloped, allow upside down, auto=false]{$\cong$}}},
		 Isom/.style={above,every to/.append style={edge node={node [sloped, allow upside down, auto=false]{\raisebox{-0.7 ex}[0ex][0ex]{\smash{\ensuremath{\sim}}}}}}},
		 Isom_inv/.style={above,every to/.append style={edge node={node [sloped, allow upside down, auto=false]{\raisebox{0.7 ex}[0ex][0ex]{\smash{\ensuremath{\sim}}}}}}}
		 }

\newcommand{\xrightarrowdbl}[2][]{%
\xrightarrow[#1]{#2}\mathrel{\mkern-14mu}\rightarrow
}
\newcommand\xrightarrowtail[2][]{\ensurestackMath{\mathrel{%
\stackengine{1pt}{%
  \stackengine{0pt}{\rightarrowtail}{\scriptstyle#2}{O}{c}{F}{F}{S}%
}{\scriptstyle#1}{U}{c}{F}{F}{S}%
}}}
\newcommand{\eqarrow}{\xrightarrow{\raisebox{-0.7 ex}[0ex][0ex]{\smash{\ensuremath{\sim}}}}}
\newcommand{\bicart}[1][0]{\ar[rd, phantom, "\square", xshift=#1]}

\theoremstyle{definition}
\newtheorem{definition}{Definition}[section]
\newtheorem{theorem}{Theorem}
\newtheorem*{theorem*}{Theorem}
\newtheorem{notation}[definition]{Notation}
\newtheorem{corollary}[definition]{Corollary}
\newtheorem*{corollary*}{Corollary}
\newtheorem{lemma}[definition]{Lemma}
\newtheorem{prop}[definition]{Proposition}
\newtheorem*{acknowledgements}{Acknowledgements}
\newtheorem{warning2}[definition]{Warning}
\newtheorem{remark}[definition]{Remark}
\newtheorem{example}[definition]{Example}

\newcommand{\HA}[1]{\cite[#1]{lurie_higher_2017}\xspace}
\newcommand{\CIS}[1]{\cite[#1]{cisinski_higher_2019}\xspace}
\newcommand{\HTT}[1]{\cite[#1]{lurie_higher_2009}\xspace}
\newcommand{\SAG}[1]{\cite[#1]{lurie_spectral_nodate}\xspace}

\newcommand{\eqdot}{.}
\newcommand{\eqcomma}{,}

\newcommand{\D}{\mathcal{D}}
\providecommand{\C}{}
\renewcommand{\C}{\mathcal{C}}
\renewcommand{\H}{\mathcal{H}}
\newcommand{\A}{\mathcal{A}}
\newcommand{\B}{\mathcal{B}}
\newcommand{\E}{\mathcal{E}}
\newcommand{\F}{\mathcal{F}}

\newcommand{\N}{\mathbb{N}}
\newcommand{\Z}{\mathbb{Z}}

\newcommand{\K}{\mathcal{K}}
\newcommand{\R}{\mathcal{R}}

\newcommand{\SW}{\mathcal{SW}}

\newcommand{\ModA}{\operatorname{Mod}(\A)}
\newcommand{\Mod}{\operatorname{Mod}(\ho(\E))}
\newcommand{\modA}{\operatorname{mod}(\A)}
\newcommand{\Eff}{\operatorname{eff}(\ho(\E))}
\newcommand{\EffA}{\operatorname{eff}(\A)}

\newcommand{\dbmod}{D^b(\ModA)}
\newcommand{\dbsmod}{D^b(\modA)}

\newcommand{\kba}{\K^b(\A)}
\newcommand{\kbe}{\K^b(\E)}
\newcommand{\dbe}{D^b(\E)}

\newcommand{\cat}{\texorpdfstring{$\infty$}{infinity}-category\xspace}
\newcommand{\cats}{\texorpdfstring{$\infty$}{infinity}-categories\xspace}
\newcommand{\categorical}{\texorpdfstring{$\infty$}{infinity}-categorical\xspace}
\newcommand{\onecat}{\texorpdfstring{$1$}{1}-category\xspace}

\newcommand{\sh}{\mathcal{H}^\text{st}}
\newcommand{\she}{\sh(\E)}
\newcommand{\shge}{\sh_{\geq 0}(\E)}

\newcommand{\PSf}{\Pres_{\Sigma, f}(\E)}
\newcommand{\PS}{\Pres_\Sigma(\E)}
\newcommand{\PSA}{\Pres_\Sigma(\A)}
\newcommand{\PSfA}{\Pres_{\Sigma, f}(\A)}

\newcommand{\Fun}{\operatorname{Fun}}
\newcommand{\cofib}{\operatorname{cofib}}
\newcommand{\coker}{\operatorname{coker}}
\newcommand{\fib}{\operatorname{fib}}

\newcommand{\Pres}{\mathcal{P}}
\newcommand{\PSift}{\Pres_\Sigma}

\newcommand{\heart}{^\heartsuit}

\newcommand{\core}{\operatorname{core}}
\newcommand{\colim}{\operatorname{colim}}
\newcommand{\const}{\operatorname{Const}}

\providecommand{\U}{}
\renewcommand{\U}{\mathcal{U}}
\newcommand{\V}{\mathcal{V}}

\newcommand{\ho}{\operatorname{Ho}}
\newcommand{\op}{^{op}}

\newcommand{\id}{id}
\newcommand{\map}{\operatorname{Map}}

\newcommand{\catab}{\ensuremath{\mathbf{Ab}}\xspace}

\newcommand{\catspaces}{\ensuremath{\mathbf{S}}\xspace}
\newcommand{\catspectra}{\ensuremath{\mathbf{Sp}}\xspace}
\newcommand{\catstable}{\ensuremath{\mathbf{St}_{\infty}}\xspace}
\newcommand{\catexact}{\ensuremath{\mathbf{Ex}_{\infty}}\xspace}
\newcommand{\simcatexact}{\ensuremath{\mathbf{Ex}^\Delta_{\infty}}\xspace}

\newcommand{\ses}{x \xrightarrowtail{i} y \xrightarrowdbl{p} z}

\author{Jona Klemenc}
\email{jona.klemenc@uni-bonn.de}

\title[The stable hull of an exact infinity-category]{The stable hull of an exact $\infty$-category}

\begin{document}
\begin{abstract}
We construct a left adjoint $\mathcal{H}^\text{st}\colon \mathbf{Ex}_{\infty} \rightarrow \mathbf{St}_{\infty}$
to the inclusion $\mathbf{St}_{\infty} \hookrightarrow \mathbf{Ex}_{\infty}$
of the $\infty$-category of stable $\infty$-categories into the $\infty$-category
of exact $\infty$-categories, which we call the stable hull.
For every exact $\infty$-category $\mathcal{E}$, the unit functor
$\mathcal{E} \rightarrow \mathcal{H}^\text{st}(\mathcal{E})$ is fully faithful and preserves and reflects
exact sequences.
This provides an $\infty$-categorical variant of the Gabriel--Quillen embedding 
for ordinary exact categories. 
If $\mathcal{E}$ is an ordinary exact category,
the stable hull $\mathcal{H}^\text{st}(\mathcal{E})$ is equivalent
to the bounded derived $\infty$-category of $\mathcal{E}$.
\end{abstract}
\maketitle
\tableofcontents
\section*{Introduction}
Every abelian category has a canonical structure of an ordinary exact category
given by the class of all short exact sequences.
Conversely, every ordinary small exact category admits an embedding into an abelian category with good properties,
the Gabriel--Quillen embedding.
\begin{theorem*}[{\cite[Th.~A.7.1]{thomason_higher_2007}}]\label{gabriel_quillen_embedding}
	Let $\E$ be a small exact category. Then, there is an abelian category $\A$
	and a fully faithful exact functor $\E \rightarrow \A$ that reflects exact sequences.
	Moreover, $\E$ is closed under extensions in $\A$.
\end{theorem*}
As extension-closed subcategories of abelian categories are exact categories,
this gives an alternative description of exact categories as extension-closed subcategories
of abelian categories.

From the point of view of derived categories, it is natural to work with the structure
of triangulated categories, which do not have good categorical properties. In recent
years, Lurie proposed an enhancement for triangulated categories called \emph{stable \cats} \cite{lurie_stable_2009}.

Exact \cats were introduced by Barwick in \cite{barwick_exact_2015} as a generalization of
ordinary exact categories in the sense of Quillen \cite{quillen_higher_1973}. Small exact \cats together with exact functors between them can be organized into
an \cat, \catexact. The \cat \catexact contains as a full subcategory both
\begin{itemize}[noitemsep,topsep=0pt]
	\item the category of ordinary small exact categories and exact functors between them and
	\item the \cat \catstable of small stable \cats and exact functors between them.
\end{itemize}
In this article we construct a functor 
\begin{equation*}
    \sh\colon \catexact \rightarrow \catstable
\end{equation*}
called the
\textit{stable hull} functor.
The main result of this article is the following theorem,
whose proof is given at the end of \Cref{sec:props_of_stable_hull}.
\begin{theorem}\label{thm:main}
    The functor $\sh$ is left adjoint to the inclusion 
    $\catstable \hookrightarrow \catexact$.
    Moreover, for every exact \cat $\E$, the unit functor
    $\eta_{\E}\colon \E \rightarrow \she$
    \begin{enumerate}
        \item\label{item:eta_ff}is fully faithful,
        \item\label{item:eta_preserves}preserves and reflects exact sequences and
        \item\label{item:eta_closed_under_ext} the essential image of $\E$ in $\she$ is closed under extensions.
    \end{enumerate}
\end{theorem}
In fact, for every exact \cat $\E$ and every stable \cat $\C$, restriction along the unit functor $\E \rightarrow \she$
induces an equivalence of \cats
\[
    \Fun^\text{ex}(\she, \C) \eqarrow \Fun^\text{ex}(\E, \C)
\]
between the full subcategories of the functor \cats spanned by those functors which are exact.
The study of a closely related universal property of the bounded derived category of an ordinary exact category
was initiated by Keller using \textit{epivalent towers} \cite{keller_derived_1991} and later
advanced by Porta \cite{porta_universal_2017} using the language of derivators.
Recently, Bunke, Cisinski, Kasprowski and Winges showed that for an ordinary exact category $\E$,
the canonical functor $\E \rightarrow \dbe$ into the bounded derived \cat of $\E$ is the universal exact functor into
a stable \cat \cite[Cor.~7.59]{bunke_controlled_2019}.
Their result readily implies the existence of a canonical equivalence
    \begin{equation}\label{eq:equivalence_bounded_derived}
        \dbe \eqarrow \she \eqdot
    \end{equation}
Note that if $\E$ is moreover abelian, a universal property of $\dbe$ as a stable \cat with a t-structure
was established by Antieau, Gepner and Heller \cite[Prop.~3.26]{antieau_k-theoretic_2019}.
In the context of exact categories, t-structures are in general not available.

\Cref{thm:main} and the equivalence in \eqref{eq:equivalence_bounded_derived} show that unit functor $\eta_{\E}\colon \E \rightarrow \she$ provides both
\begin{itemize}[noitemsep,topsep=0pt]
	\item an \categorical variant of the Gabriel--Quillen embedding and
	\item a generalization of the bounded derived \cat of an ordinary exact category
	to the more general class of exact \cats.
\end{itemize}
Furthermore, in light of \cite[Ex.~3.5]{barwick_exact_2015}, \Cref{thm:main} immediately yields the following characterization of exact \cats:

\begin{corollary*}
    Let $\E$ be a small additive \cat and let $\E_\dagger$, $\E^\dagger$ be subcategories of $\E$.
    The following are equivalent:
    \begin{enumerate}
        \item The triple $(\E, \E_\dagger, \E^\dagger)$ is an exact \cat.
        \item There exists a stable \cat $\C$ and a fully faithful functor $\E \hookrightarrow \C$ 
            such that
            \begin{enumerate}
                \item the essential image of $\E$ is closed under extensions in $\C$,
                \item a morphism in $\E$ lies in $\E_\dagger$ if and only if its cofiber in $\C$
                lies in the essential image of $\E$ and
                \item a morphism in $\E$ lies in $\E^\dagger$ if and only if its fiber in $\C$
                lies in the essential image of $\E$.
            \end{enumerate}
    \end{enumerate}
\end{corollary*}
\section{Preliminaries}
We use freely the language of \cats as developed in \cite{lurie_higher_2009,lurie_higher_2017}.
\subsection{Stable and prestable \cats}
First proposed by Lurie \cite{lurie_stable_2009}, stable \cats provide an enhancement of
triangulated categories in the sense of Verdier \cite{verdier_categories_nodate}.
\begin{definition}[\HA{Prop.~1.1.1.9}]
	An \cat $\C$ is \textit{stable} if it has the following properties:
	\begin{enumerate}
		\item The \cat $\C$ is pointed, that is it admits a zero object.
		\item The \cat $\C$ admits finite limits and finite colimits.
		\item A square in $\C$ is cocartesian if and only if it is cartesian.
	\end{enumerate}
\end{definition}
\begin{prop}[\HA{Th.~1.1.2.14}]
	Let $\C$ be a stable \cat. The homotopy category $\ho(\C)$ is additive and
	has a canonical triangulated structure.
\end{prop}
\begin{definition}[\HA{Def.~1.2.1.4}]
	Let $\C$ be a stable \cat. A \textit{t-structure on $\C$} is a t-structure on the homotopy
    category $\ho(\C)$.
    More precisely, if $\C$ is equipped with a t-structure, we denote  the full subcategories of $\C$ spanned by those objects belonging to
	$(\ho(\C))_{\geq n}$ and $(\ho(\C))_{\leq n}$ by $\C_{\geq n}$ and
	$\C_{\leq n}$, respectively.
\end{definition}
\begin{prop}[\HA{Prop.~1.2.1.5, Prop.~1.2.1.10, Rem.~1.2.1.1}]
	Let $\C$ be a stable \cat with a t-structure. Then the
	following statements hold:
	\begin{enumerate}
		\item For all $n \in \N$, the inclusion $\C_{\geq n} \hookrightarrow \C$ admits
			a right adjoint $\tau_{\geq n}\colon  \C \rightarrow \C_{\geq n}$.
		\item For all $n \in \N$, the inclusion $\C_{\leq n} \hookrightarrow \C$ admits
			a left adjoint $\tau_{\leq n}\colon  \C \rightarrow \C_{\leq n}$.
		\item For all $m, n \in \N$, there is a natural equivalence
		\begin{equation*}
			\tau_{\leq m} \circ \tau_{\geq n} \simeq \tau_{\geq n} \circ \tau_{\leq m} \eqdot
		\end{equation*}
		\item The heart $\C\heart \vcentcolon= \C_{\leq 0} \cap \C_{\geq 0}$ is equivalent to the nerve of
			its homotopy category, which is an abelian category.
    \end{enumerate}
    We denote the composite $\tau_{\leq 0} \circ \tau_{\geq 0}$ by $\pi_0\colon \C \rightarrow \C\heart$.
\end{prop}
\begin{definition}[\SAG{Constr.~C.1.1.1}]
	Let $\D$ be a pointed \cat which admits finite colimits
	and let $\Sigma\colon \D \rightarrow \D$ be the suspension functor, given on objects by
	$\Sigma X = 0 \sqcup_X 0$.
	We denote the direct limit of the tower
	\begin{equation*}
		\D \xrightarrow{\Sigma} \D \xrightarrow{\Sigma} \D \xrightarrow{\Sigma} \cdots
	\end{equation*}
	by $\SW(\D)$ and refer to it as the \emph{Spanier--Whitehead \cat of $\D$}.
	It comes with a canonical functor $\D \rightarrow \SW(\D)$.
\end{definition}
\begin{remark}
	As can be seen from the definition, the \cat $\SW(\D)$
	is generated by the essential image of $\D$ under finite suspension.
\end{remark}
\begin{prop}[\SAG{Prop.~C.1.1.7}]
	Let $\D$ be a pointed \cat which admits finite colimits. 
	The \cat $\SW(\D)$ is a stable \cat and, for every stable \cat $\C$, restriction along
	the functor $\D \rightarrow \SW(\D)$ induces
    an equivalence
    \begin{equation*}
        \Fun^\text{ex}(\SW(\D), \C) \eqarrow \Fun^\text{rex}(\D, \C) \eqdot
    \end{equation*}
    between the full subcategory of $\Fun(\SW(\D), \C)$ spanned by the exact functors
    and the full subcategory of $\Fun(\D, \C)$ spanned by the right exact functors.
\end{prop}
\begin{definition}[\SAG{Def.~C.1.2.1, Prop.~C.1.2.2}]
	Let $\D$ be a pointed \cat admitting finite colimits. The following conditions are equivalent:
	\begin{enumerate}
		\item The canonical functor $\D \rightarrow \SW(\D)$ is fully faithful and the essential image
		is closed under extensions.
		\item The \cat $\D$ is equivalent to a full subcategory of a stable \cat $\C$
		which is closed under extensions and finite colimits.
    \end{enumerate}
    If $\D$ has these properties, we call $\D$ a \textit{prestable \cat}.
\end{definition}
\begin{example}\label{prestable_subcat_ex}
    Let $\D$ be a prestable \cat, and let $\D'$ be a full subcategory of $\D$ which is closed under extensions and finite colimits.
    Then $\D'$ is a prestable \cat and the canonical exact functor $i\colon \SW(\D') \rightarrow \SW(\D)$ is exact.
    As $i$ restricts to the fully faithful composite $\D' \hookrightarrow \D \hookrightarrow \SW(\D)$ and $\SW(\D')$ is generated by
    $\D'$ under finite suspensions, $i$ is fully faithful.
\end{example}
\begin{definition}
	An \cat is \emph{additive} if it has finite products and finite coproducts and its homotopy category is additive.
\end{definition}
\begin{definition}
    Let $\A$ be a small additive \cat. We write $\PSA$ for the full subcategory of
    the presheaf category of spaces spanned by those presheaves which preserve finite products.
	The \cat $\PSA$ is often called the \textit{nonabelian derived \cat of $\A$}.
\end{definition}
\begin{notation}
	Let $\B$ be a small additive \onecat. We define
    \begin{equation*}
        \operatorname{Mod}(\B) \vcentcolon= \Fun^\pi(\B\op, \catab) \eqdot
    \end{equation*}
\end{notation}
The following proposition is one of the key inputs of this article.
\begin{prop}[\SAG{Prop.~C.1.5.7}]\label{key_input}
    Let $\A$ be a small additive \cat.
	Restriction along the infinite loops functor
	\begin{equation*}
		\Omega^\infty\colon \catspectra^\text{cn} \rightarrow \catspaces
	\end{equation*}
	from the \cat of connective spectra to the \cat of spaces induces an equivalence of \cats
	\begin{equation*}
		\Fun^\pi(\A\op, \catspectra^\text{cn}) \eqarrow \Fun^\pi(\A\op, \catspaces)
	\end{equation*}
	between the full subcategories of the functor categories spanned by those functors which
	preserve finite products.

	As a consequence, the \cat $\PSA$ is prestable and there exists a t-structure on $\SW(\PSA)$ such that
    $\SW(\PSA)_{\geq 0}$ is the essential image of $\PSA$ in $\SW(\PSA)$.
    Under this t-structure, the heart of $\SW(\PSA)$ identifies with $\operatorname{Mod}(\ho(\A))$.
    In particular, for all $F \in \PSA$, $x \in \A$, the homotopy group $\pi_0(F(x))$ is 
    canonically isomorphic to $(\pi_0(F))(x)$.
\end{prop}
\subsection{Exact \cats}
\begin{definition}[\cite{barwick_exact_2015}]
	An exact \cat is a triple {$(\E, \E_\dagger, \E^\dagger)$} where $\E$ is an additive \cat and
	$\E_\dagger$, $\E^\dagger$ are subcategories of $\E$ such that the following conditions hold:
	\begin{enumerate}
        \item Every morphism whose domain is a zero object is in $\E_\dagger$. 
            Every morphism whose codomain is a zero object is in $\E^\dagger$.
		\item Pushouts of morphisms in $\E_\dagger$ exist and every pushout of a morphism in $\E_\dagger$ is
			in $\E_\dagger$; pullbacks of
			morphisms in $\E^\dagger$ exist and every pullback of a morphism in $\E^\dagger$ is in $\E^\dagger$.
		\item\label{item:ambigressive_square} For a square in $\E$ of the form
			\begin{equation*}
			\begin{tikzcd}
			x \arrow[r, "f"] \arrow[d, "g"] & y \arrow[d, "g'"] \\
			x' \arrow[r, "f'"] & y' \eqcomma
			\end{tikzcd}
			\end{equation*}
			the following conditions are equivalent:
			\begin{enumerate}
				\item The square is cocartesian, $f \in \E_\dagger$ and $g \in \E^\dagger$.
				\item The square is cartesian, $f' \in \E_\dagger$ and $g' \in \E^\dagger$.
			\end{enumerate}
	\end{enumerate}
    We call the morphisms in $\E_\dagger$ \textit{cofibrations} and, in diagrams, mark them
    by tails. We call the morphisms in $\E^\dagger$ \emph{fibrations} and, in diagrams, mark them by double heads.
\end{definition}
In the following, we suppress $\E_\dagger, \E^\dagger$ from the notation.
\begin{definition}
    Let $\E$ be an exact \cat. An \emph{exact sequence} is a bicartesian square in $\E$ of the form
	\begin{equation}\label{diag:exact_seq}
	\begin{tikzcd}
		x \arrow[r, tail, "i"] \bicart \arrow[d, two heads] &y \arrow[d, two heads, "p"] \\
		0 \arrow[r, tail] &z
	\end{tikzcd}
    \end{equation}
    such that $i$ is a cofibration and $p$ is a fibration.
    We sometimes write $\ses$ for such a sequence. We call $z$ the \textit{cofiber of $i$}
	and $x$ the \textit{fiber of $p$}.
	We sometimes write $\cofib(i)$ for $z$ and $\fib(p)$ for $x$.
\end{definition}
\begin{example}\label{ex:ordinary_exact_is_exact}
    Let $\E$ be an ordinary exact category in the sense of Quillen. The exact structure on $\E$ endows (the nerve of) $\E$
    with the structure of an exact \cat, and vice versa.
\end{example}
\begin{example}\label{ex:stable_is_exact}
    Let $\C$ be an additive \cat. The triple $(\C, \C, \C)$ is an exact \cat if and only if $\C$ is stable.
\end{example}
\begin{example}[cf.~{\cite[Ex.~3.5]{barwick_exact_2015}}]
	Let $\D$ be a prestable \cat and denote by $D'$ the subcategory of $\D$ spanned
	by morphisms $p$ occuring in bicartesian squares of the form \eqref{diag:exact_seq}.
	Then $(\D, \D, \D')$ is an exact \cat.
\end{example}
\section{The \cat of finite additive presheaves} \label{section:A_first_embedding}
\begin{definition}\label{def:cbge}
	Let $\E$ be a small exact \cat. We denote by $\PSf$ the smallest subcategory of $\PS$ 
	which contains $\E$ and is closed under finite colimits.
\end{definition}
\begin{prop}
	The \cat $\PSf$ admits finite colimits and, for every \cat
	$\D$ which admits finite colimits, restriction along the inclusion $\E \hookrightarrow \PSf$ induces an equivalence
	\begin{equation*}
		\Fun^\text{rex}(\PSf, \D) \eqarrow \Fun^\Sigma(\E, \D)
	\end{equation*}
	between the full subcategory of $\Fun(\PSf, \D)$ spanned
	by those functors which are right exact and 
	the full subcategory of $\Fun(\E, \D)$ spanned by those functors which
	preserve finite coproducts.
\end{prop}
\begin{proof}
	By taking the smallest subcategory of $\PS$ which contains $\E$ and is closed under finite colimits,
	we follow the construction of the \cat with the claimed universal property given in the proof of \HTT{Prop.~5.3.6.2}.
\end{proof}
\begin{notation}
	Let $f\colon x \rightarrow y$ be a morphism in $\E$. We denote its image in $\PSf \subseteq \PS$ under
	the Yoneda embedding by $\hat{f}\colon \hat{x} \rightarrow \hat{y}$.
	Similarly, we denote its image in $\Mod$ under the composite
	$\E \hookrightarrow \PS \xrightarrow{\pi_0} \Mod$
	by
	$\overline{f}\colon \overline{x} \rightarrow \overline{y}$.
\end{notation}
\begin{prop}\label{prop:E_is_projective}
	Let $x \in \E$ and $F \in \PS$. The homotopy group $\pi_0(\map(\hat{x}, \Sigma F))$ is trivial.
\end{prop}
\begin{proof}
	By the Yoneda Lemma, there is an isomorphism of abelian groups
	\begin{equation*}
		\pi_0(\map(\hat{x}, \Sigma F)) \cong \pi_0((\Sigma F)(x)) \eqdot
	\end{equation*}
	The codomain identifies with $\pi_0(\Sigma F)(x)$; this is
	a trivial group as $F$ is in the aisle of the standard t-structure on $\SW(\PS)$.
\end{proof}
\begin{prop}\label{prop:psf_prestable}
	The \cat $\PSf$ is prestable.
\end{prop}
\begin{proof}
	As, by definition, $\PSf$ is closed under finite colimits inside $\PS$,
	it suffices to prove that $\PSf$ is closed under extensions in $\PS$. Equivalently, we prove
	that the cofiber of every map $Z \rightarrow \Sigma X$
	with $Z, X \in \PSf$ lies in $\Sigma \PSf \subseteq \Sigma \PS$. 
	We show this by proving the following: For every object $X \in \PSf$ the class
	of objects $Z \in \PSf$ satisfying this property
	contains the essential image of the Yoneda embedding and
	is closed under
	finite colimits in $\PS$.
	The former condition follows directly from \Cref{prop:E_is_projective}.

	For the latter, the duals of \cite[Th.~7.3.27, Prop.~7.3.28]{cisinski_higher_2019}
	imply that it is sufficient
	to show that this collection of objects is closed under pushouts.

	Consider a pushout square $\sigma$ in $\PSf$
	\begin{equation*}
		\begin{tikzcd}
			Z \ar[r, "g"] \ar[d, "h"] & Z' \ar[d, "h'"] \\
			Z'' \ar[r, "g'"] & Y
		\end{tikzcd}
	\end{equation*}
	such that $Z$, $Z'$ and $Z''$ have this property.
	We need to show that the cofiber of every morphism $f\colon Y \rightarrow \Sigma X$ 
	is in $\Sigma \PSf$, so let such a morphism be given.
	We extend $f$ to a morphism of diagrams
	$f'\colon \sigma \rightarrow \const_{\Sigma X}$,
	where the codomain is the diagram constant in $\Sigma X$.
	Because colimits commute, the pushout of $f'$, taken in the \cat of squares in $\PS$,
	along a morphism
	with codomain zero  is a pushout diagram of the form
	\begin{equation*}
		\begin{tikzcd}
			\cofib(f \circ h' \circ g) \ar[r] \ar[d] & \cofib(f \circ h') \ar[d] \\
			\cofib(f \circ g') \ar[r] & \cofib(f) \eqdot
		\end{tikzcd}
	\end{equation*}
	As $\PSf$ is closed under finite colimits, this implies that $\cofib(f)$ is in $\PSf$,
	finishing the proof.
\end{proof}
\begin{warning2}\label{warning:PSf}
	We have the following diagram of fully faithful functors, which commutes
	up to natural equivalence:
	\begin{equation*}
	\begin{tikzcd}
		\PSf \arrow[r, hook] \arrow[d, hook] & \PS \arrow[d, hook] \\
		\SW(\PSf) \arrow[r, hook]            & \SW(\PS) \eqdot       
		\end{tikzcd}
	\end{equation*}
	While $\PS$ embeds into $\SW(\PS)$ as the aisle of a t-structure,
	$\PSf$ need not be the aisle of a t-structure on $\SW(\PSf)$,
	as $\PSf$ might not have finite limits. Furthermore, the heart $\SW(\PS)\heart$ is not contained in $\PSf$. However,
	we can still consider the homotopy groups of objects in $\PSf$
	relative to the t-structure on $\SW(\PS)$---these might just not lie in
	$\PSf$. Similarly, objects in $\PSf$ can be truncated as objects in $\PS$---but again,
	the resulting object does not have to be inside $\PSf$.
\end{warning2}
\subsection{Primitive acyclic objects}
\begin{definition}
	Let $i: x \rightarrow y$ be a morphism in $\E$ and $z$ be an object in $\E$.
	A morphism $\cofib(\hat{i}) \rightarrow z$ is called a \textit{primitive quasi-isomorphism}
	if there exists an exact sequence $\ses$ such that $f$ is the
	canonical morphism induced by the sequence. The colimit in $\PSf$ of the diagram
	\begin{equation}\label{diag:primtive_acyclic}
		\begin{tikzcd}[column sep=small, row sep=tiny]
			\hat{x}\ar[rr, "\hat{i}"] \ar{dd}\ar[rd] &&\hat{y}\ar{dd}\ar[rd, "\hat{p}"]\\
			&0\ar[crossing over]{rr}&&\hat{z}\\
			0\ar[rr] \ar[rd] &&0\\
			&0\ar[crossing over, leftarrow]{uu}
		  \end{tikzcd}
	\end{equation}
	is a \textit{primitive acyclic} object.
\end{definition}
\begin{prop}\label{primitive_acyclic_old_def}
	The colimit of a diagram in $\PSf$ of the form \eqref{diag:primtive_acyclic} is canonically
	equivalent to the cofiber of the primitive quasi-isomorphism $\cofib(\hat{i}) \rightarrow z$ in $\PSf$.
\end{prop}
\begin{proof}
	This is a direct application of \HTT{Prop.~4.4.2.2} to the diagram \eqref{diag:primtive_acyclic}
	and the decomposition
	\begin{equation*}
		\begin{tikzcd}[
			column sep=1.0em,
			row sep=0.6em,
			nodes=overlay,
			transform shape, nodes={scale=0.7}
			]
			\bullet \ar[rrrr] \ar[rrdd] &&&&\bullet & &&&& & \bullet \ar[rrrr] \ar[rrdd] &&&& \bullet \ar[rrdd]\\
			&&&&&\ar[rrrr, hookrightarrow] &&&& \phantom{.} \\
			&& \bullet &&&& & &&&& & \bullet \ar[rrrr] &&&&\circ \\
			&& \ar[dddd, hookrightarrow] \\ \\ \\ \\ && \phantom{.} \\
			\bullet \ar[rrrr] \ar[rrdd] \ar[dddd] &&&&\bullet \ar[dddd] \\ \\
			&& \bullet  \\ \\
			\circ \ar[rrrr] \ar[rrdd] &&&&\circ \\ \\
			&& \circ \ar[uuuu, crossing over, leftarrow]
		\end{tikzcd}
	\end{equation*}
	of the cube.
\end{proof}
\begin{definition}
	An object $M \in \Mod$ is \textit{effaceable} if it is the cokernel of a morphism
	$\overline{y} \xrightarrow{\overline{p}} \overline{z}$, where
	$p\colon y \xrightarrowdbl{} z$ is a fibration in $\E$.
	We denote the full subcategory of $\Mod$ spanned by the effaceable functors by $\Eff$.
\end{definition}
The next proposition
ensures we can do a lot of the work in functor categories between ordinary categories.
\begin{prop}\label{remark:primitive_acyclic_effaceable}
	The primitive acyclic objects
	constitute the essential image of the composition
	$\Eff \hookrightarrow \Mod \simeq \SW(\PS)\heart \hookrightarrow \PSift$.
	In particular, this essential image lies in $\PSf$.
\end{prop}
\begin{proof}
	Let $A \in \PSf$ be a primitive acyclic object corresponding to an exact sequence
	$\ses$. We claim that $A$ is the effaceable object corresponding to $\overline{p}$.

	Let $Z$ be the cofiber of $\hat{i}$ in $\PSf$. We show that the induced map
	\begin{equation*}
		\pi_n(Z) \rightarrow \pi_n(\hat{z})
	\end{equation*}
	is an isomorphism for $n > 0$, and that in the sequence
	\begin{equation*}
		\begin{tikzcd}
			\bar{x} \arrow[r, "\bar{i}"] & \bar{y} \arrow[r, "f"] & \pi_0(Z) \arrow[r, "g"] & \overline{z} \eqcomma
		\end{tikzcd}
	\end{equation*}
	the map $f$ is an epimorphism and the map $g$ is a monomorphism. The statement then follows
	from the long exact sequence
	\begin{equation*}
		\dots \rightarrow \pi_{n+1}(A) \rightarrow \pi_n(Z) \rightarrow \pi_n(\hat{z})
		\rightarrow \pi_n(A) \rightarrow \pi_{n-1}(Z) \rightarrow \cdots \eqdot
	\end{equation*}
	To this end, note that all of these claims can be checked pointwise. Fix an
	object $t \in \E$.
	There is a commutative diagram 
	where the rows are the Serre long exact sequences
	\begin{equation*}
		\begin{tikzcd}[column sep=small]
		\dots \arrow[r] \arrow[d] & {\pi_0(\map(\hat{t}, \hat{x}))} \arrow[d, equal] \arrow[r] & {\pi_0(\map(\hat{t}, \hat{y}))} \arrow[d, equal] \arrow[r] & {\pi_0(\map(\hat{t}, Z)} \arrow[r] \arrow[d] & \pi_{-1}(\map(\hat{t}, \hat{x})) \arrow[d] \\
		\dots \arrow[r]           & {\pi_0(\map(t, x))} \arrow[r]           	   & {\pi_0(\map(t, y))} \arrow[r]                  & {\pi_0(\map(t, z))} \arrow[r]                           & C          
		\end{tikzcd}
	\end{equation*}
	whose bottom row is induced by the fiber sequence $\ses$ in $\E$ and whose top row is
	induced by the fiber sequence $\hat{x} \rightarrow \hat{y} \rightarrow Z$
	in $\SW(\PS)$.
	\Cref{prop:E_is_projective} implies that $\pi_{-1}(\map(\hat{t}, \hat{x})) = 0$.
	The claim follows from the Four Lemma and the Five Lemma.
\end{proof}
In the next two lemmas, we prove that $\Eff$ is \textit{weakly Serre} as a subcategory
of $\Mod$.
\begin{lemma}\label{prop:eff_closed_ext}
	The full subcategory of effaceable functors is closed under extensions in $\Mod$.
\end{lemma}
\begin{proof}
	Let $K \rightarrow L \rightarrow M$ be a short exact sequence in $\Mod$ such that $K$ and $M$
	lie in $\Eff$. By definition, we can extend this to a diagram with exact columns
	\begin{equation*}
	\begin{tikzcd}
	&\overline{y_1} \arrow[d, "\overline{p_1}"] & & \overline{y_2} \arrow[d, "\overline{p_2}"] \\
	&\overline{z_1} \arrow[d, "c_1"] & & \overline{z_2} \arrow[d, "c_2"] \\
	0 \arrow[r] & K \arrow[r] & L \arrow[r] & M \arrow[r]& 0 \eqcomma 
	\end{tikzcd}
	\end{equation*}
	where
	$p_1$, $p_2$ are fibrations in $\E$
	and $c_1$, $c_3$ are epimorphisms in $\Mod$. Since objects in the essential image of the Yoneda
	embedding are projective in $\Mod$, the Horseshoe Lemma yields
	a commutative diagram with exact columns
	\begingroup 
	\setlength\arraycolsep{1.5pt}
		\begin{equation*}
		\begin{tikzcd}
			&\overline{y_1} \arrow[d, "\overline{p_1}"] \arrow[r]& \overline{y_1} \oplus \overline{y_2} \arrow[r] \arrow[d, "\begin{pmatrix}
				\overline{p_1} & 0\\
				\overline{f} & -\overline{p_2}
			\end{pmatrix}", dashrightarrow]& \overline{y_2} \arrow[d, "\overline{p_2}"] \\[2em]
			&\overline{z_1} \arrow[d, "c_1"] \arrow[r ]& \overline{z_1} \oplus \overline{z_2} \arrow[r] \arrow[d, "c_3"] & \overline{z_2} \arrow[d, "c_2"] \\
			0 \arrow[r] & K \arrow[r] & L \arrow[r] & M \arrow[r]& 0 \eqcomma
		\end{tikzcd} 
	\end{equation*}
	\endgroup
	where $c_3$ is an epimorphism. 
	\Cref{prop:diagram_lemma_1} guarantees that the morphism indicated by a dashed arrow is a fibration in $\E$.
	Hence $L$ is effaceable.
\end{proof}
\begin{lemma}\label{prop:eff_closed_kernel_cokern}
	The full subcategory of effaceable functors is closed under kernels and cokernels in $\Mod$.
\end{lemma}
\begin{proof}
	Let $g\colon L \rightarrow M$ be morphism in $\Eff$. We denote the kernel and the cokernel in $\Mod$ by 
	$f\colon K \rightarrow L$ and $h\colon M \rightarrow N$, respectively. 
	Again, since the essential image of $\ho(\E)$ in $\Mod$ consists of projective objects,
	we can construct a commutative diagram with exact columns
	\begin{equation}\label{diag:closed_under_kernels_small}
		\begin{tikzcd}
			&                  & \overline{y_1} \arrow[d, "\overline{p_1}"] \arrow[r, "\overline{b}"] & \overline{y_2} \arrow[d, "\overline{p_2}"] &             &   \\
			&                  & \overline{z_1} \arrow[d, "c_2"] \arrow[r, "\overline{c}"]            & \overline{z_2} \arrow[d, "c_3"]            &             &   \\
		0 \arrow[r] & K \arrow[r, "f"] & L \arrow[r, "g"]                                     & M \arrow[r, "h"]                           & N \arrow[r] & 0
		\end{tikzcd}
	\end{equation}
	where $p_1$ and $p_2$ are fibrations in $\E$,
	and $c_2$ and $c_3$ are epimorphisms in $\Mod$.
	Here, projectiveness of $\overline{z_2}$ and the fact that $c_3$ is an epimorphism means that $g \circ c_2$ factors
	through $c_3$ as $\overline{c}$.
	The composition $\overline{c} \circ \overline{p_1}$ is factors through the kernel of $c_3$,
	so by exactness of the right column and projectiveness of $\overline{y_2}$ it factors through $\overline{p_2}$ as $\overline{b}$.
	The commutative square
	\begin{equation*}
		\begin{tikzcd}
			y_1 \arrow[r, "b"] \arrow[d, two heads, "p_1"] &y_2 \arrow[d, two heads, "p_2"] \\
			z_1 \arrow[r, "c"] &z_2
		\end{tikzcd}
	\end{equation*}
	in $\ho(\E)$ can be lifted to a commutative square in $\E$.
	By Proposition \ref{prop:extend_exact_diagram}, this square can be completed further to a commutative diagram
	\begin{equation*}
	\begin{tikzcd}
		x_1 \ar[r, "a"] \ar[d, tail, "i_1"] \bicart &x_2 \ar[d, tail, "i_3"] \ar[r, equal] &x_2 \ar[d, tail, "i_2"]\\
		y_1 \ar[r, "b_1"] \ar[d, two heads, "p_1"] &d \ar[r, "b_2"] \ar[d, two heads, "p_3"] \bicart  &y_2 \ar[d, two heads, , "p_2"] \\
		z_1 \ar[r, equal] &z_1 \ar[r, "c"] &z_2
	\end{tikzcd}
	\end{equation*}
	where all columns are exact sequences and the marked squares are bicartesian.
	The commutativity of the bottom right square implies that the composition
	$c_2 \circ \overline{p_3}$ in $\Mod$ factors through $f$, say as $f \circ c_1$.
	Hence we can extend the diagram \eqref{diag:closed_under_kernels_small} to a larger diagram
	\begingroup 
	\setlength\arraycolsep{1.7pt}
	\begin{equation}\label{diag:closed_under_cokernels_large}
	\begin{tikzcd}[column sep=large]
	& \overline{y_1} \oplus \overline{x_2} \arrow[d, "\begin{pmatrix}
		\overline{b_1} & \overline{i_3}
	\end{pmatrix}"'] \arrow[r, "\begin{pmatrix}
		\id & 0
	\end{pmatrix}"] & \overline{y_1} \arrow[d, "\overline{p_1}"] \arrow[r, "\overline{b}"] & \overline{y_2} \arrow[d, "\overline{p_2}"] \arrow[r, "\begin{pmatrix}
		0 \\
		\id
	\end{pmatrix}"] & \overline{z_1} \oplus \overline{y_2} \arrow[d, "\begin{pmatrix}
		\overline{c} & \overline{p_2}
	\end{pmatrix}"] \\
	&\overline{d} \arrow[r, "\overline{p_3}"] \arrow[d, "c_1"] & \overline{z_1} \arrow[d, "c_2"] \arrow[r, "\overline{c}"]            & \overline{z_2} \arrow[d, "c_3"] \ar[r, equal]            & \overline{z_2}\arrow[d, "h \circ c_3"] \\
	0 \ar[r] & K \arrow[r, "f"] & L \arrow[r, "g"] & M \arrow[r, "h"]  & N \ar[r] & 0\eqdot                           
	\end{tikzcd} 
	\end{equation}
	\endgroup
	Exactness of the right column is immediate and $h \circ c_3$ is an epimorphism as it is the composition of two epimorphisms.
	Note that for every exact sequence
	$p \xrightarrowtail{} q \xrightarrowdbl{} r$ in $\E$, the corresponding Serre sequence implies that the sequence
	\begin{equation*}
		\overline{p} \rightarrow \overline{q} \rightarrow \overline{r}
	\end{equation*}
	is exact at $\overline{q}$. 
	Applying this observation to the exact sequence $x_2 \xrightarrowtail{i_3} d \xrightarrowdbl{p_3} z_1$,
	a short diagram chase shows that the sequence
	\begin{equation*}
		\begin{tikzcd}[column sep=small]
			\overline{y_1} \oplus \overline{x_2} \arrow[rrrr, "\begin{pmatrix}
				\overline{b_1} & \overline{i_3}
			\end{pmatrix}"] &&&&\overline{d} \ar[r] & \coker(\overline{p_1})
		\end{tikzcd}
	\end{equation*}
	is exact at $\overline{d}$, which shows that the left column in \eqref{diag:closed_under_cokernels_large} is exact.
	Note that \cite[Lem.~4.7]{barwick_exact_2015} implies that there
	exists an exact sequence in $\E$ of the form
	\begin{equation*}
		\begin{tikzcd}[column sep=small]
			d \ar[rrr, "\begin{pmatrix}
				b_2 \\
				-p_3
			\end{pmatrix}", tail] &&& z_1 \oplus y_2 \ar[rrrr, "\begin{pmatrix}
				c & p_2
			\end{pmatrix}", two heads] &&&&z_2 \eqdot
		\end{tikzcd}
	\end{equation*}
	Again, a short diagram chase shows that
	the existence of the above sequence implies that the sequence
	\begin{equation*}
		\begin{tikzcd}[column sep=small]
			\overline{d} \ar[rr, "p_3"] && \overline{z_1} \ar[r] & \coker(\overline{p_2})
		\end{tikzcd}
	\end{equation*}
	is exact at $\overline{z_1}$, which shows that $c_1$ is an epimorphism in $\Mod$.

	By \cite[Lem.~4.7]{barwick_exact_2015}, the morphisms $\begin{pmatrix}
		b_1 & i_3
	\end{pmatrix}$ and $\begin{pmatrix}
		c & p_2
	\end{pmatrix}$
	are fibrations in $\E$. Hence the left and the right column show that $K$ and $N$ are effaceable, respectively.
\end{proof}
The following corollary is a standard result for weakly Serre
subcategories.
\begin{corollary}\label{cor:two_out_of_three_primitive_acyclics}
	Let
	\begin{equation*}
		M_1 \rightarrow M_2 \rightarrow M_3 \rightarrow M_4 \rightarrow M_5
	\end{equation*}
	be an exact sequence in $\Mod$ such that all objects but $M_3$ are effaceable.
	Then $M_3$ is effaceable as well.
\end{corollary}
\begin{proof}
	This sequence decomposes as
	\begin{equation*}
	\begin{tikzcd}[column sep=tiny, row sep=small]
		&&                           &                &                           & N_2 \arrow[rd] &               &     \\
	M_1 \arrow[rr] && M_2 \arrow[rd] \arrow[rr] &                & M_3 \arrow[ru] \arrow[rr] &                & M_4 \arrow[rr] && M_5 \\
		&&                           & N_1 \arrow[ru] &                           &                &               &    
	\end{tikzcd}
	\end{equation*}
	where $N_1$ is the cokernel of $M_1 \rightarrow M_2$, $N_2$ is the kernel of
	$M_4 \rightarrow M_5$ and the sequence
	\begin{equation*}
		N_1 \rightarrow M_3 \rightarrow N_2
	\end{equation*}
	is short exact in $\Mod$. Propositions \ref{prop:eff_closed_ext} and
	\ref{prop:eff_closed_kernel_cokern} imply that $M_3$ is effaceable.
	\end{proof}
\subsection{Comparison with Krause's derived Auslander formula}
Let $\A$ be a small abelian category.
Krause \cite{krause_deriving_2015} proves that acyclic complexes in
$\A$ of the form
\begin{equation*}
	\dots \rightarrow 0 \rightarrow X_{2} \rightarrow X_1 \rightarrow X_{0} \rightarrow 0 \rightarrow \dots
\end{equation*}
constitute the essential image of the composition
\begin{equation*}
	\operatorname{eff}(\A) \hookrightarrow \operatorname{mod}(\A) \simeq
		D^b(\operatorname{mod}(\A))\heart \hookrightarrow D^b(\operatorname{mod}(\A))
\end{equation*}
under the canonical equivalence
\begin{equation*}
	D^b(\operatorname{mod}(\A)) \simeq \kba \eqdot
\end{equation*}
\Cref{remark:primitive_acyclic_effaceable} is a generalization of this observation:
There exists a diagram of \cats which commutes up to natural equivalence
\begin{equation*}
	\begin{tikzcd}[
		column sep={1.8cm,between origins},
		row sep=tiny
		]
		\SW(\PSfA) \arrow[dddd, hook]\arrow[rrrr, Isom, leftarrow] \ar[rrd, hookleftarrow] &&&& \kba \arrow[dd, Isom] \\
		&& \A \arrow[rru, hook] \arrow[d, hook] \\[0.35cm]
		& \makebox[\widthof{shs}][r]{$\EffA$} \arrow[r, hook] & \modA \arrow[rr, hook, "\heartsuit"] \arrow[d, hook] && \dbsmod \arrow[dd, hook] \\[0.35cm]
		&& \ModA \arrow[rrd, hook,"\heartsuit"]\\
		\SW(\PSA) \arrow[rru, hookleftarrow, "\heartsuit"] \arrow[rrrr, Isom, leftarrow]  &&&& \dbmod \eqcomma
		\end{tikzcd}
\end{equation*} 
where all functors are fully faithful.
Note that \cite[Prop.~7.55]{bunke_controlled_2019} together with \HTT{Prop.~1.3.3.14} implies the existence of the horizontal
equivalences and the commutativity up to natural equivalence.

\subsection{Acyclic objects}
\begin{definition}\label{def:acyclics}
	The objects in the stable closure of the primitive acyclic objects in $\SW(\PSf)$---equivalently,
	the objects in $\SW(\PSf)$ arising from the primitive acyclic objects by positive and negative suspensions and
	cofibers---are \textit{acyclic} objects.
\end{definition}
\begin{prop}\label{prop:acyclics_built_from_primitive_acyclics}
	Let $X \in \SW(\PSf)$ be an acyclic object. For all $n \in \N$,
	the homotopy group $\pi_n(X)$ is effaceable.
\end{prop}
\begin{proof}
	As $\SW(\PSf)$ is a subcategory of $\SW(\PS)$ which is closed under positive and negative suspensions and cofibers,
	it is sufficient to show
	that the subcategory of $\SW(\PS)$ spanned by the objects with this property is closed under positive and negative suspensions and 
	cofibers, and contains the primitive acyclic objects. That this subcategory is closed under positive and negative suspensions is clear.
	The fact that
	for a primitive acyclic object $X$, the homotopy groups are effaceable is shown in
	\Cref{remark:primitive_acyclic_effaceable} (note that $\pi_n(X) = 0$ for $n \neq 0$ in this case). To prove that for a map $f\colon X \rightarrow Y$ 
	where $X, Y$ have effaceable homotopy groups, the cofiber has this property as well, consider the 
	corresponding long exact sequence of homotopy groups and apply
	\Cref{cor:two_out_of_three_primitive_acyclics}.
\end{proof}
\begin{prop}\label{prop:acyclics_right_adjoint}
	Let $X \in \SW(\PSf)$ be an acyclic object. The truncation $\tau_{\geq 0}X \in \PS$ is
	acyclic and lies in $\PSf$.
\end{prop}
\begin{proof}
	The homotopy groups of $\tau_{\geq 0}X \in \PS$ agree with those of $X$ in positive degrees
	and are zero in negative degrees. By \Cref{prop:acyclics_built_from_primitive_acyclics},
	its homotopy groups are
	all effaceable. Consider the homotopy groups of $\tau_{\geq 0}X$ as objects in $\PS$
	via the embedding $\Mod \lhook\joinrel\xrightarrow{\heartsuit} \PS$.
	Because it is bounded and connective, $\tau_{\geq 0} X$ is generated under repeated extensions
	of positive suspensions of its homotopy groups. As $\PSf$ is closed in $\PS$ under colimits and extensions,
	\Cref{remark:primitive_acyclic_effaceable} finishes the proof.
\end{proof}
\section{The stable hull}\label{section:Stable_hull}
\begin{definition}
	A morphism in $\SW(\PSf)$ is a \textit{quasi-isomorphism} if its cofiber is an acyclic object. We denote the  Dwyer-Kan localization of
	$\SW(\PSf)$ at the class of quasi-isomorphisms by $\she$. Similarly, we denote the Dwyer-Kan localization of $\PSf$
	at the class of quasi-isomorphisms in $\PSf$ by $\shge$.
\end{definition}
\begin{remark}
	The Dwyer-Kan localization of \cats, sometimes also just called \textit{\categorical localization}, is the universal functor sending a class of edges to equivalences.
	This is not equivalent to the concept described by Lurie in \HTT{Sec.~5.2.7}.
	For a very thorough treatment of the Dwyer-Kan localization, see \CIS{Sec.~7.1}.
\end{remark}
\begin{remark}
	The functor $\SW(\PSf) \rightarrow \she$ is the Verdier quotient of $\SW(\PSf)$ by
	its full subcategory of acyclic objects. It is well known that Verdier quotients of triangulated categories
	allow for a simple description of their hom-spaces.
	A very explicit description of Verdier quotients of stable \cats has been
	found by Nikolaus and Scholze in \cite[Th.~I.3.3]{nikolaus_topological_2018}, which we
	will quickly repeat here for $\she$:
	For two objects $X, Y \in \SW(\PSf)$ with images $\overline{X}, \overline{Y}$
	in $\she$, the mapping space is given by the filtered colimit
	\begin{equation*}
		\map(\overline{X}, \overline{Y}) \simeq \colim_{f: Z \rightarrow Y\text{ a quasi-isomorphism}}(X, Z)
	\end{equation*}
	The mapping spaces of Verdier quotients of stable \cats are described by
	Nikolaus and Scholze in \cite[Th.~I.3.3]{nikolaus_topological_2018}.
\end{remark}
\begin{prop}\label{prop:qi_finitely_saturated}
	The quasi-isomorphisms endow $\SW(\PSf)$ with the structure of an \cat of cofibrant objects,
	as in the dual of \CIS{Def.~7.5.7}, where all morphisms are cofibrations. The quasi-isomorphisms which lie
	in $\PSf$ endow $\PSf$ with the structure of an \cat of cofibrant objects, where all morphisms are cofibrations.
\end{prop}
\begin{proof} 
	That the axioms given in \CIS{Def.~7.4.12} are fulfilled follows readily
	from the following three properties of quasi-isomorphisms:
	\begin{itemize}
		\item \textit{Equivalences are quasi-isomorphisms:}
		The cofibers of equivalences are the zero objects.
		\item \textit{Quasi-isomorphisms are stable under pushouts:}
		Every morphism has the same cofiber as its pushout along every other morphism.
		\item \textit{Quasi-isomorphisms have the 2-out-of-3-property:}
		This can be checked on the homotopy categories, where it follows from the 
		octahedral axiom: if $f$ and $g$ are composable, then the cofiber of $g \circ f$ is an 
		extension of the cofiber of $f$ and the cofiber of $g$. \qedhere
	\end{itemize}
\end{proof}
We obtain the following corollary as an application of \CIS{Prop.~7.5.11}.
\begin{corollary}\label{prop:cofibrant_localization}
	Let $\D$ be an \cat with finite colimits.  Restrictions along the functors $\SW(\PSf) \rightarrow \she$ and
	$\PSf \rightarrow \shge$ induce the following equivalences of \cats:
	\begin{align*}
	\begin{split}
			\Fun^\text{rex}(\she, \D) &\eqarrow \Fun^\text{rex}_\text{q-i}(\SW(\PSf), \D) \eqcomma \\
			\Fun^\text{rex}(\shge, \D) &\eqarrow \Fun^\text{rex}_\text{q-i}(\PSf, \D) \eqdot
	\end{split}
	\end{align*}
	Here, the subscript indicates that we restrict further to the full
	subcategory of those functors which send quasi-isomorphisms to equivalences.
\end{corollary}
\begin{prop}\label{prop:quasi_iso_right_adjoint}
	Let $X \in \PSf$, $Y \in \SW(\PSf)$, and $w\colon X \rightarrow Y$ be a quasi-isomorphism.
	Using the conventions of Remark \ref{warning:PSf}, the truncation
	$\tau_{\geq 0} Y \in \PS$
	lies in $\PSf$ and the induced map
	$X \rightarrow \tau_{\geq 0} Y$ is a
	quasi-isomorphism.
\end{prop}
\begin{proof}
	In the following proof, we view $w$ as a morphism in $\SW(\PS)$ by means of the fully faithful functor $\SW(\PSf) \hookrightarrow \SW(\PS)$.
	Note that by the 2-out-of-3 property of quasi-isomorphisms and since the counit
	$\tau_{\geq 0}X \rightarrow X$ is an equivalence, it suffices to prove
	the claim for $\tau_{\geq 0} w$. Applying $\tau_{\geq 0}$ to the cofiber diagram of $w$
	yields a commutative diagram
	\begin{equation*}
		\begin{tikzcd}
			\tau_{\geq 0}X \arrow[r, "\tau_{\geq 0} w"] \arrow[d] & \tau_{\geq 0}Y \arrow[d]\arrow[r]& 0 \arrow[d] \\
			0 \arrow[r] & \tau_{\geq 0}Z \arrow[r] &  \tau_{\geq 0}\Sigma X \eqcomma
		\end{tikzcd}
	\end{equation*}
	where $Z$---and hence by \Cref{prop:acyclics_right_adjoint},
	$\tau_{\geq 0} Z$ as well---is acyclic by assumption.
	Both the right and the outer square are bicartesian before truncating, hence they are pullback diagrams in $\PS$,
	as truncation preserves limits. Because $\PS$ is prestable and $\tau_{\geq 0} \Sigma X$ is equivalent
	to $\Sigma X$, they are bicartesian in $\PS$ and hence bicartesian in $\SW(\PS)$ as well.
	By the Pasting Law for cartesian squares \HTT{Lem.~4.4.2.1}, the left square is bicartesian as well.
	Since $\PSf$ is closed in $\PS$ under extensions, this finishes the proof.
\end{proof}
\begin{prop}\label{prop:connective_inclusion_ff}
	The functor $\shge \rightarrow \she$, which is induced by the composite $\PSf \hookrightarrow 
	\SW(\PSf) \rightarrow \she$, is fully faithful.
\end{prop}
\begin{proof}
	By Corollary \ref{prop:cofibrant_localization}, the induced functor is right exact. Hence, by
	the dual of \CIS{Th.~7.6.10}, it is sufficient to check that the induced functor on homotopy categories is
	fully faithful. Note that the homotopy category of the \categorical Dwyer-Kan localization agrees with
	the 1-categorical localization of the homotopy category and that the morphism sets of the localization can be calculated using left calculus of fractions.
	We need to prove that, for
	objects $X, Y \in \PSf$, the induced map
	\begin{equation*}
		\hom_{\ho(\shge)}(X, Y) \rightarrow \hom_{\ho(\she)}(X, Y)
	\end{equation*}
	is an isomorphism. To prove that this map is injective, note that two roofs in $\ho(\PSf)$ which become equivalent in $\ho(\she)$
	yield a commutative diagram in $\ho(\SW(\PSf))$ of the form
	\begin{equation*}
		\begin{tikzcd}
			& Y' \arrow[Isom]{d} \ar[rd, Isom, leftarrow]\\
			X \arrow[ru] \arrow[r] \arrow[rd] & \widetilde{Y} \ar[d, Isom, leftarrow] \ar[r, Isom, leftarrow] & Y\\
			& Y'' \ar[ru, Isom, leftarrow]                                
		\end{tikzcd}
	\end{equation*}
	where all the marked morphisms are quasi-isomorphisms and all the
	objects but $\widetilde{Y}$ are in $\ho(\PSf)$.
	Replacing $\widetilde{Y}$ with $\tau_{\geq 0}\widetilde{Y}$ and using Proposition
	\ref{prop:quasi_iso_right_adjoint}, we can form an analogous diagram in $\ho(\PSf)$, showing the roofs
	are equivalent as maps in $\ho(\shge)$. To show that this map is surjective,
	using the same argument, every roof in $\ho(\SW(\PSf))$ is equivalent to a roof in $\ho(\PSf)$.
\end{proof}
\begin{corollary}\label{cor:whitehead_dbge}
	The \cat $\shge$ is prestable and the induced functor $\SW(\shge) \rightarrow \she$ is an equivalence.
\end{corollary}
\begin{proof}
	The \cat $\she$ is stable by \cite[Th.~I.3.3]{nikolaus_topological_2018}.
	By Corollary \ref{prop:cofibrant_localization}, the embedding $\shge \hookrightarrow \she$ is right exact. 
	We first show that $\shge$ is closed in $\she$ under extensions, which proves both
	that $\shge$ is prestable and that the functor $\SW(\shge) \rightarrow \she$ is fully faithful, see \Cref{prestable_subcat_ex}.

	Let
	$X \rightarrow Y \rightarrow Z$ be an exact sequence in $\she$ with $X, Z$ in the essential image of $\shge$.
	After shifting, we are left to show that the cofiber of the induced map $Z \rightarrow \Sigma X$ lies
	in the essential image of $\Sigma \shge \subseteq \shge$. Since $\shge \rightarrow \she$ is fully
	faithful, we can assume the map $Z \rightarrow \Sigma X$ to be in $\shge$. Using the left calculus
	of fractions, such a map comes up to equivalence from a map in $\PSf$.
	The claim now
	follows from the facts that $\PSf$ is a prestable \cat and that the functor $\PSf \rightarrow \shge$ is right exact.

	It remains to show that the functor $\SW(\shge) \rightarrow \she$ is essentially surjective.
	For this, note that
	for every object $X \in \she$, there exists some $n \in \Z$ such that $\Sigma^n X$ is equivalent to the image
	of some $Y \in \PSf$ under the composition $\PSf \rightarrow \SW(\PSf) \rightarrow \she$. But then
	$X$ is equivalent to the image of $\Sigma^{-n} Y \in \SW(\shge)$.
\end{proof}
\subsection{Size considerations}
As passing from $\E$ to $\PS$ removes any smallness assumptions $\E$ might have,
it is \emph{a priori} not clear which smallness properties are preserved by the construction $\E \rightarrow \she$.
In this section, we see that
$\she$ is essentially $\V$-small.

\begin{lemma}\label{prop:SWPSf_locally_small}
    Let $\U \subseteq \V$ be a Grothendieck universe such that $\E$ is a locally
    $\U$-small \cat. Then $\SW(\PSf)$ is locally $\U$-small.
\end{lemma}
\begin{proof}
    By virtue of \CIS{Cor.~5.7.9} and because
    \begin{equation*}
        \pi_{n}(\hom(X, Y)) \cong \pi_0(\hom(\Sigma^n X, Y)) \eqcomma
    \end{equation*}
    it is sufficient to prove that, for every pair of objects
    $X, Y \in \SW(\PSf)$, the set $\pi_0(\hom(X, Y))$ is $\U$-small.
    Fix $x \in \E$ and consider the full subcategory of $\SW(\PSf)$ spanned 
    by those objects $Y$ for which $\pi_0(\hom(\hat{x}, \Sigma^n Y))$ is $\U$-small for all $n \in \Z$.
    This subcategory is clearly closed under suspensions and extensions and contains $\E$ by assumption
    and \Cref{prop:E_is_projective}.
    As $\SW(\PSf)$ is generated by
    $\E$ under suspensions and extensions, this shows that
    $\pi_0(\hom(\hat{x}, Y))$
    is $\U$-small for all $x \in \E, Y \in \PSf$.

    Consider the full subcategory of $\PSf$ spanned by those objects $X$ for 
    which $\pi_0(\hom(X, Y))$ is
    $\U$-small for all $Y \in \PSf$.
    By an analogous argument, this subcategory is $\PSf$ itself, which finishes the proof.
\end{proof}

\begin{lemma}\label{lemma:PSf_preserves_smallness}
    The Spanier--Whitehead \cat $\SW(\PSf)$ is essentially $\V$-small.
\end{lemma}
\begin{proof}
    By \CIS{Prop.~5.7.6} and \Cref{prop:SWPSf_locally_small}, it suffices to show that the set of
    isomorphism classes is $\V$-small. This is clear, as $\SW(\PSf)$ is generated by $\E$
    under finite suspensions and extensions.
\end{proof}

\begin{prop}
    The stable \cat $\she$ is essentially $\V$-small.
\end{prop}
\begin{proof}
    As $\pi_n(\hom(X, Y)) \cong \pi_0(\hom(\Sigma^n X, Y))$, this may be checked
    on the homotopy category, which is the Verdier quotient of $\ho(\SW(\PSf))$ by the 
    acyclics. As the Verdier quotient of a small triangulated category is small \cite[Prop.~2.2.1]{neeman_triangulated_2001},
    this follows directly from \Cref{lemma:PSf_preserves_smallness}.
\end{proof}

\begin{remark}
    The
    construction $\E \rightarrow \she$ does \textit{not} preserve local smallness.
    Indeed, in \Cref{prop:onecat_she} we show that in the case of ordinary exact categories,
    there is a canonical equivalence $\dbe \eqarrow \she$
    and the construction $\E \rightarrow \dbe$ does not preserve smallness \cite[Ex.~6.2.A]{freyd_1964}.
\end{remark}
\subsection{Properties of the stable hull}\label{sec:props_of_stable_hull}
First, we establish the universal property of $\she$, which implies that
the construction $\E \rightarrow \she$ can be organized into a left adjoint to the inclusion
$\catstable \hookrightarrow \catexact$. 
 \begin{lemma}\label{prop:local_pqi_qi}
	Every object $T \in \PS$ which is local to the primitive quasi-isomorphisms is local to all quasi-isomorphisms.
\end{lemma}
\begin{proof}
	Let $T \in \PS$ be local to the primitive quasi-isomorphisms. Then, the mapping space $\map(A, \Sigma T)$ is 
	trivial for every primitive acyclic object $A$: 
	Let $\ses$ be the exact sequence in $\E$ corresponding to a fixed primitive acyclic object
	$A$. Then there is a Serre long exact sequence
	\begin{equation*}
		\dots \rightarrow \pi_n(\map(\hat{z}, \Sigma T)) \rightarrow
		\pi_n(\cofib(i), \Sigma T) \rightarrow
		\pi_{n - 1}(\map(A, \Sigma T)) \rightarrow
		\cdots \eqdot
	\end{equation*}
	The morphism $\pi_n(\map(\hat{z}, \Sigma T)) \rightarrow \pi_n(\map(\cofib(i), \Sigma T))$ is an isomorphism for 
	$n > 0$ by assumption, while $\pi_0(\map(\hat{z}, \Sigma T))$ is trivial by Proposition 
	\ref{prop:E_is_projective}.

	Now using
	\begin{itemize}
		\item Proposition \ref{prop:acyclics_built_from_primitive_acyclics},
		\item the fact that in $\PS \subseteq \SW(\PS)$, every bounded object arises as an iterated extension of its homotopy groups, and
		\item that the class of objects $Y \in \PS$ for which
			$\map(Y, \Sigma X)$ is trivial is closed under extensions,
	\end{itemize}
	we conclude that 
	$\map(X, \Sigma T)$ is trivial for every acyclic object $X \in \PSf$.
	As the cofiber of every quasi-isomorphism $w\colon Y \rightarrow Z$ is acyclic, the claim follows
	from the corresponding Serre long exact sequence
	\begin{equation*}
		\dots \rightarrow \pi_n(\map(Z, T)) \rightarrow
		\pi_n(\map(Y, T)) \rightarrow
		\pi_{n-1}(\map(\cofib(w), T)) \rightarrow \cdots \eqdot \qedhere
	\end{equation*}
\end{proof}
\begin{lemma}\label{lemma:ps_send_primitive_to_isos_send_all_to_isos}
	Let $\D$ be an \cat which admits small colimits. A cocontinuous functor 
	$\PS \rightarrow \D$ sends primitive quasi-isomorphisms to equivalences if and only if it sends all
	quasi-isomorphisms to equivalences.
\end{lemma}
\begin{proof}
	Let
	\begin{equation*}
		\begin{tikzcd}
			\PS \ar[r, shift left=0.7ex]  \ar[r, hookleftarrow, shift right=0.7ex, "i"']&L_q(\PS)
		\end{tikzcd}
	\end{equation*}
	be the cocontinuous localization of $\PS$ along the
	quasi-isomorphisms. As $\PS$ is presentable, $i$ is the inclusion of the full subcategory
	of $\PS$ spanned by those objects which are local to quasi-isomorphisms.
	Hence, by \Cref{prop:local_pqi_qi}, this localization is also the cocontinuous localization of
	$\PS$ along the class of the primitive quasi-isomorphisms.
	Therefore, a cocontinuous functor $\PS \rightarrow \D$ sends quasi-isomorphisms to equivalences
	if and only if it is equivalent to the restriction of some cocontinuous functor $L_q(\PS) \rightarrow \D$
	if and only if it sends primitive quasi-isomorphisms to equivalences.
\end{proof}
\begin{prop}\label{prop:psf_send_primitive_to_isos_send_all_to_isos}
	Let $\D$ be an \cat which admits finite colimits. A right exact functor
	$\PSf \rightarrow \D$ sends primitive quasi-isomorphisms to equivalences if and only if it sends all
	quasi-isomorphisms to equivalences.
\end{prop}
\begin{proof}
	By embedding $\D$ into $\Pres(\D\op)\op$, we can assume $\D$ to admit all small colimits. Then,
	both right exact functors $\PSf \rightarrow \D$ and colimit-preserving functors
	$\PS \rightarrow \D$ correspond to functors $\E \rightarrow \D$ which preserve finite coproducts.
	Therefore,
	the inclusion $\PSf \hookrightarrow \PS$ induces an equivalence
	\begin{equation*}
		\Fun^!(\PS, \D)) \eqarrow \Fun^\text{rex}(\PSf, \D) \eqdot
	\end{equation*}
	Here $\Fun^!(\PS, \D)$ denotes the full subcategory of $\Fun(\PS, \D))$ spanned by the cocontinuous functors. 
	The conclusion now follows directly from \Cref{lemma:ps_send_primitive_to_isos_send_all_to_isos}.
\end{proof}
\begin{prop}\label{prop:universal_prop_dbge}
	Let $\D$ be an \cat which admits finite colimits. Write $\Fun^{\Sigma, \text{seq}}(\E, \D)$ for the full subcategory
	of $\Fun(\E, \D)$ spanned by those functors which preserve finite coproducts and send exact sequences
	to pushout squares. Then, restriction along the composition
	\begin{equation*}
		\E \hookrightarrow \PSf \rightarrow \shge
	\end{equation*}
	induces an equivalence of \cats
	\begin{equation*}
		\Fun^\text{rex}(\shge, \D) \eqarrow \Fun^{\Sigma, \text{seq}}(\E, \D) \eqdot
	\end{equation*}
\end{prop}
\begin{proof}
	By \Cref{prop:cofibrant_localization}, restriction along the functor $\PSf \rightarrow \shge$ induces an
	equivalence of \cats
	\begin{equation*}
		\Fun^\text{rex}(\shge, \D) \eqarrow
			\Fun^\text{rex}_\text{q-i}(\PSf, \D)
	\end{equation*}
	As we have shown in \Cref{prop:psf_send_primitive_to_isos_send_all_to_isos}, the codomain of this equivalence agrees
	with
		$\Fun^\text{rex}_\text{p q-i}(\PSf, \D)$, the full subcategory of $\Fun(\PSf, \D)$
		spanned by the right exact functors
		which send primitive quasi-isomorphisms to equivalences. Also, the 
	equivalence from Proposition \ref{def:cbge} restricts to an equivalence
	\begin{equation*}
		\Fun^\text{rex}_\text{p q-i}(\PSf, \D) \eqarrow
			\Fun^{\Sigma, \text{seq}}(\E, \D) \eqdot
	\end{equation*}
	Composing these equivalences yields the result.
\end{proof}
\begin{prop}\label{prop:she_ff}
	The functor $\E \rightarrow \she$ is fully faithful.
\end{prop}
\begin{proof}
	Let $\K$ be the class of all finite simplicial sets,
	and let $\R$ be the union of
	\begin{itemize}
		\item the class of exact sequences in $\E$ and
		\item the class of diagrams in $\E$ exhibiting finite coproducts.
	\end{itemize}
	\Cref{prop:universal_prop_dbge} shows that there is an equivalence $\shge \eqarrow \Pres_\R^\K$
	commuting with the inclusion of $\E$, where $\E \rightarrow \Pres_\R^\K$
	is the universal functor
	into a \cat with $\K$-indexed colimits which sends diagrams in $\R$ to colimit diagrams, see \HTT{Sec.~5.3.6}.
	As the functor $\E \rightarrow \Pres_\R^\K$ is fully faithful, so are the functor $\E \rightarrow \shge$ and the composition
	$\E \rightarrow \shge \hookrightarrow \she$.
\end{proof}
\begin{definition}[{\cite[Def.~4.1]{barwick_exact_2015}}]
	Let $\E$, $\F$ be exact \cats. We say a functor $\E \rightarrow \F$ is \emph{exact} if
	it preserves cofibrations, fibrations, zero objects, pushouts along cofibrations and pullbacks along fibrations.
\end{definition}
\begin{remark}
	As can be easily checked by considering split exact sequences,
	an exact functor between exact \cats preserves biproducts.
\end{remark}
\begin{prop}\label{prop:when_is_functor_exact}
	Let $\E$, $\F$ be exact \cats.
	A functor $\E \rightarrow \F$ which preserves finite coproducts is exact if and only if it
	sends exact sequences to exact sequences.
\end{prop}
\begin{proof}
	By \cite[Prop.~4.8]{barwick_exact_2015}, it is sufficient to prove that a pushout square
	in $\E$ of the form
	\begin{equation*}
	\begin{tikzcd}
	A \arrow[r] \arrow[d, tail] & A' \arrow[d, tail] \\
	B \arrow[r] &B'
	\end{tikzcd}
	\end{equation*}
	is sent to a pushout square in $\F$.
	By Proposition \ref{prop:big_diagram_lemma}, such a pushout square fits into a 
	commutative diagram
	\begin{equation*}
		\begin{tikzcd}
			A \arrow[r] \arrow[d, tail] & A' \arrow[d, tail] \arrow[r] & 0 \arrow[d] \\
			B \arrow[r] &B \arrow[r, two heads] & C
		\end{tikzcd}
	\end{equation*}
	such that both the outer rectangle and the right square are bicartesian.
	By assumption, after applying the functor $\E \rightarrow \F$ these squares are bicartesian;
	applying Proposition \ref{prop:big_diagram_lemma} again finishes the proof.
\end{proof}
\begin{corollary}\label{cor:inclusion_exact}
	The functor $\E \hookrightarrow \she$ is exact and the functor $\E \hookrightarrow \shge$ preserves exact sequences.
\end{corollary}
\begin{proof}
	Because the three functors $\E \hookrightarrow \PSf$, $\PSf \rightarrow \SW(\PSf)$ and $\SW(\PSf) \rightarrow \she$ preserve
	finite coproducts, so does their composition.
	Hence, it suffices to prove that $\E \hookrightarrow \shge$ preserves exact sequences;
	that the functor $\E \hookrightarrow \she$ is exact
	then follows from \Cref{prop:when_is_functor_exact}.
	We need to show for that for an exact 
	sequence $\ses$,
	the induced map
	\begin{equation*}
		\cofib(\hat{i}) \rightarrow z
	\end{equation*}
	in $\shge$ is an equivalence. This map can be chosen to be the image
	of the corresponding map in $\PSf$, which is a (primitive) quasi-isomorphism.
\end{proof}
We are ready to establish the universal property of the stable hull.
\begin{prop}\label{prop:adjointness_prop_of_she}
	Let $\E$ be an exact \cat $\E$. For every stable \cat $\C$,
	restriction along the functor $\E \hookrightarrow \she$ induces an equivalence of \cats
	\begin{equation*}
		\Fun^\text{ex}(\she, \C) \rightarrow \Fun^\text{ex}(\E, \C)
	\end{equation*}
	between the full subcategories of the functor \cats spanned by those functors which are exact.
\end{prop}
\begin{proof}
	Using the notation of Proposition \ref{prop:universal_prop_dbge}, the 
	codomain of the functor agrees with
	$\Fun^{\Sigma, \text{seq}}(\E, \C)$ by Proposition \ref{prop:when_is_functor_exact}.
	The functor $\E \hookrightarrow \she$ induces equivalences of \cats
	\begin{equation*}
		\Fun^\text{ex}(\she, \C) \eqarrow \Fun^\text{rex}(\shge, \C) \eqarrow 
			\Fun^{\Sigma, \text{seq}}(\E, \C) \eqdot
	\end{equation*}
	The rightmost equivalence comes from the
	\Cref{prop:universal_prop_dbge}, while the leftmost one comes from
	\Cref{cor:whitehead_dbge} together with the universal property of the Spanier--Whitehead
	construction.
\end{proof}
The following definition is taken from \cite[Not.~4.2]{barwick_exact_2015}.
\begin{definition}
	We let \simcatexact be the following simplicial category:
	The objects of \simcatexact are small exact \cats
	and, for two small exact \cats $\E$, $\F$, we set
	\begin{equation*}
		\simcatexact(\E, \F) \vcentcolon= \core(\Fun^\text{ex}(\E, \F)) \eqcomma
	\end{equation*}
	where $\Fun^\text{ex}(\E, \F)$ is the full subcategory of $\Fun(\E, \F)$ spanned by the exact functors.
	We let $\catexact$ be the homotopy coherent nerve of $\simcatexact$ and $\catstable$ be the full subcategory of $\catexact$
	spanned by the stable \cats (with every morphism marked both as a fibration and a cofibration as in Example \ref{ex:stable_is_exact}).
\end{definition}
\begin{corollary}\label{cor:left_adjoint_existence}
	The embedding $\catstable \hookrightarrow \catexact$ admits a left adjoint
	\begin{equation*}
		\sh\colon \catexact \rightarrow \catstable
	\end{equation*}
	The unit of this adjunction is, for every exact \cat
	$\E$, equivalent to the canonical functor $\E \rightarrow \she$.
\end{corollary}
\begin{proof}
	By \HTT{Prop.~5.2.7.8},
	it is sufficient to show that for every object $\E \in \catexact$,
	there exists a morphism
	$\E \rightarrow \H$ in $\catexact$ such that $\H$ is in $\catstable$ and 
	for every stable \cat $\C$, the induced map
	\begin{equation*}
		\map_{\catexact}(\H, \C) \rightarrow \map_{\catexact}(\E, \C)
	\end{equation*}
	is an isomorphism in the homotopy category of spaces.
	By the construction of \catexact, this map is equivalent to the induced map
	\begin{equation*}
		\core(\Fun^\text{ex}(\H, \C)) \rightarrow \core(\Fun^\text{ex}(\E, \C)) \eqdot
	\end{equation*}
	\Cref{prop:adjointness_prop_of_she} shows that $\E \hookrightarrow \she$ has the desired property.
\end{proof}
In \cite{keller_chain_1990} Keller provides an elementary proof of the Gabriel--Quillen embedding.
The proof of the following proposition is an \categorical analogue of his proof that
this embedding reflects exact sequences and is closed under exact sequences. 
\begin{prop}\label{prop:she_reflects_preserves}
	The functor $\E \hookrightarrow \she$ reflects exact sequences and the essential image of this functor is closed under extensions.
\end{prop}
\begin{proof}
	As the embedding $\shge \hookrightarrow \SW(\shge) \eqarrow \she$ reflects exact sequences and
	has an essential image which is closed under extensions, it is sufficient to
	show that the functor $\E \hookrightarrow \shge$ has these properties as well.

	Let $\hat{x} \rightarrow Y' \rightarrow \hat{z}$ be an exact sequence in $\shge$
	with $x, z \in \E$. As the functor $\E \hookrightarrow \shge$ is fully faithful,
	it suffices to show that there exists an exact sequence 
	$\ses$ in $\E$ which is equivalent to the given sequence in $\shge$.
	The hom-sets in $\ho(\shge)$ can be computed
	using the left calculus of fractions on $\ho(\PSf)$.
	Using this, one can form a commuting diagram of the form
	\begin{equation*}
		\begin{tikzcd}[
			column sep={4.5em,between origins},
			row sep={3.3em,between origins},
		  ]
			\hat{x} \arrow[r, "j"] \arrow[rd, dashed] & Y \arrow[r] \ar[d, Isom, leftarrow] &\cofib(j) \arrow[Isom, "f"']{d} \\
			& Y' \arrow[rd, dashed] & Z \ar[d, Isom, leftarrow, "g"'] \\
			&& \hat{z}
			\end{tikzcd}
	\end{equation*}
	in $\ho(\shge)$,
	where all the solid arrows come from morphisms in $\PSf$, the cofiber of $j$ is taken in $\PSf$
	and the vertical maps are all quasi-isomorphisms in $\PSf$.

	It now suffices to show that the exact sequence
	$\hat{x} \rightarrow Y \rightarrow Z$ is equivalent to the image of an exact
	sequence in $\E$. Let $d\colon Z \rightarrow Q$ be the cofiber of
	$f$ in $\PSf$. Note that,
	as $f$ is a quasi-isomorphism,
	$Q$ is acyclic. By \Cref{prop:acyclics_built_from_primitive_acyclics}, $\pi_0(Q)$ is effaceable,
	hence it is the cokernel of some fibration $\overline{q}: \overline{w} \rightarrow \overline{v}$ in $\Mod$.
	The diagram in $\Mod$
	\begin{equation*}
		\begin{tikzcd}[column sep=large]
			&& \overline{v} \ar[d] \\
			\overline{z} \ar[r, "\pi_0({g})"'] \ar[urr, dashed] & \pi_0(Z) \ar[r, "\pi_0({d})"'] & \pi_0({Q})
		\end{tikzcd}
	\end{equation*}
	has a lift, as the vertical map is an epimorphism. We denote the pullback of  $q$ along this lift by $p_*\colon y_* \xrightarrowdbl{}{} z$.
	Note that the composition $d \circ g \circ \hat{p_*}$ is the zero-map in $\ho(\PSf)$. Hence the diagram in $\PSf$
	\begin{equation*}
		\begin{tikzcd}[column sep=large]
			&& \cofib(j) \ar[d] \\
			\hat{y_*} \ar[r, "\hat{p_*}"'] \ar[urr, dashed] & \hat{z} \ar[r, "g"'] & Z
		\end{tikzcd}
	\end{equation*}
	has a lift which can, by \Cref{prop:E_is_projective}, be lifted further to a square in $\PSf$
	\begin{equation*}
		\begin{tikzcd}
			\hat{y_*} \arrow[d] \arrow[r, "\hat{p_*}"]  & \hat{z} \arrow[d, "g"]  \\
			Y \arrow[r] & Z \eqdot
			\end{tikzcd}
	\end{equation*}
	By \Cref{prop:extend_exact_diagram}, this square can be extended to a morphism of exact sequences
	of the form
 	\begin{equation*}
	 	\begin{tikzcd}
	 	\sigma \arrow[d] 	& \hat{x_*} \arrow[r] \arrow[d] \bicart  & \hat{y_*} \arrow[r, "\hat{p_*}"] \arrow[d]  & \hat{z} \arrow[d, equal] \\
	 	\sigma'' \arrow[d] 	&\hat{x} \arrow[r] \arrow[d, equal] & \hat{y} \arrow[d] \arrow[r] \bicart  & \hat{z} \arrow[d, Isom] \\
	 	\sigma' 			&\hat{x} \arrow[r,"j"]           & Y \arrow[r] & Z                    
	 	\end{tikzcd}
 	\end{equation*}
 	in which the marked squares are bicartesian. 
	 The morphism $\sigma'' \rightarrow \sigma'$ is an equivalence. Since the functor $\E \hookrightarrow \shge$ is fully
	 faithful, $\E$ has fibers of fibrations and $p_*$ is a fibration, $\hat{x_*}$ can be chosen to be the image of
	 some $x_* \in \E$. By a similar argument, the middle row can be chosen to be the image of an exact sequence in $\E$.
\end{proof}
We are ready to prove \Cref{thm:main}.
\begin{proof}[Proof of \Cref{thm:main}]
    The fact that $\sh$ is left adjoint to the inclusion $\catstable \hookrightarrow \catexact$ is established in \Cref{cor:left_adjoint_existence}, which also
    states that the unit of the adjunction is equivalent to the construction of the stable hull.
    We apply our findings about the stable hull to prove the three claimed properties of the unit functor.
    For the stable hull, we proved (\ref{item:eta_ff}) in \Cref{prop:she_ff}
    and (\ref{item:eta_preserves}) and (\ref{item:eta_closed_under_ext}) in \Cref{prop:she_reflects_preserves}.
\end{proof}

\subsection{The stable hull of an ordinary exact category}\label{sec:1_cat}

\newcommand{\kbgp}{\K^b_{\geq 0}(P)}

If $\E$ is an exact \onecat in the sense of Quillen,
its nerve inherits the structure of an exact \cat.
In this case, the universal property of the stable hull and the universal property of the bounded derived category
proven in \cite[Cor.~7.59]{bunke_controlled_2019}
imply the existence of a canonical equivalence
    \begin{equation*}
        \dbe \eqarrow \she \eqdot
    \end{equation*}

For completeness, we include a slightly different proof which
explains our naming of acyclic objects and quasi-isomorphisms.

\begin{prop}[{\cite[Prop.~7.55]{bunke_controlled_2019}}]\label{prop:cbe_is_swpsf}
    Let $\E$ be an exact \onecat. Then there exists a canonical equivalence
    \begin{equation*}
        \kbe \eqarrow \SW(\PSf)
    \end{equation*}
    whose composition with the inclusion $\E \rightarrow \kbe$ is equivalent to the Yoneda
    embedding $\E \rightarrow \PSf$.
\end{prop}
\begin{remark}
    While the definition of $\PSf$ used by Bunke, Cisinski, Kasprowski and Winges differs
    from ours, \Cref{prop:psf_prestable} ensures that they are equivalent.
\end{remark}
\begin{prop}
    Let $\E$ be an exact category, considered an exact \cat as in Example \ref{ex:ordinary_exact_is_exact}. 
    Then, the equivalence described in Proposition \ref{prop:cbe_is_swpsf} sends
    an object to an acyclic object in the sense of Definition \ref{def:acyclics}
    if and only if it is an acyclic complex in
    the usual sense. Hence, this equivalence preserves and reflects quasi-isomorphisms.
\end{prop}
\begin{proof}
    The full subcategory of acyclic chain complexes in $\kbe$ is the triangulated closure
    of acyclic complexes of the form
    \begin{equation*}
        \dots \rightarrow 0 \rightarrow X_2 \xrightarrowtail{i} X_1 \xrightarrowdbl{p} X_0
        \rightarrow 0 \rightarrow \cdots \eqdot
    \end{equation*}
    It suffices to show that, under the equivalence $\kbe \eqarrow \SW(\PSf)$,
    an acyclic complex of this form
    is sent to the primitive acyclic object
    corresponding to the exact sequence $X_2 \xrightarrowtail{i} X_1 \xrightarrowdbl{p} X_0$.
    This follows from right exactness, as both of these objects are equivalent the cofiber of the
    induced map $\cofib(\hat{i}) \rightarrow X_0$, see \Cref{primitive_acyclic_old_def}.
\end{proof}
\begin{corollary}[{\emph{cf.}~\cite[Cor.~7.59]{bunke_controlled_2019}}]\label{prop:onecat_she}
    Let $\E$ be an exact category, considered an exact \cat as in Example \ref{ex:ordinary_exact_is_exact}. Then, there is a canonical equivalence
    $\dbe \eqarrow \she$,
    whose composition with the inclusion $\E \hookrightarrow \dbe$ is equivalent to the inclusion
    $\E \hookrightarrow \she$.
\end{corollary}
\appendix
\section*{Appendix. Diagram lemmas in exact \cats}
\setcounter{section}{1}
\setcounter{definition}{0}
In this appendix, we include the proofs of three diagram lemmas we use in the article.
\begin{prop}\label{prop:big_diagram_lemma}
    Let $\E$ be an exact \cat. Consider a square in $\E$ of the form
    \begin{equation*}
        \begin{tikzcd}
            x \arrow[r, "i", tail] \arrow[d, "f"] & y \arrow[d, "g"] \\
            x' \arrow[r,"j", tail] & y'
        \end{tikzcd}
    \end{equation*}
    where the morphisms $i$ and $j$ are cofibrations. The following conditions are equivalent:
    \begin{enumerate}
		\item\label{enum:bicart} The square is bicartesian.
        \item\label{enum:pushout} The square is a pushout square.
        \item\label{enum:extend} 
        	There exists a commutative diagram
        	\begin{equation*}
        	\begin{tikzcd}
        	x \arrow[r, "i", tail] \arrow[d, "f"] & y \arrow[d, "g"] \\
        	x' \arrow[r,"j", tail]\arrow[d, two heads] & y' \arrow[d, two heads] \\
        	0 \arrow[r, tail] & z
        	\end{tikzcd}
        	\end{equation*}
        	where both the lower and the outer rectangle are bicartesian.
    \end{enumerate}
\end{prop}
\begin{proof}
	The equivalence (\ref{enum:bicart}) $\Leftrightarrow$ (\ref{enum:pushout}) is proven
	in \cite[Lemma~4.5]{barwick_exact_2015}.

	For (\ref{enum:pushout}) $\Rightarrow$ (\ref{enum:extend}), one pushes out $j$ along $x' \xrightarrowdbl{} 0$,
	and applies the Pasting Lemma for pushouts.

	We show (\ref{enum:extend}) $\Rightarrow$ (\ref{enum:pushout}).
	We first reduce to the case
	where $f$ is the identity by constructing a diagram
	\begin{equation*}
		\begin{tikzcd}[column sep=small, row sep=tiny]
			x\ar[rr, tail] \ar{dd}\ar[rd, equal] &&y\ar{dd}\ar[rd, equal]\\
			&x\ar[crossing over, tail]{rr}&&y\ar{dd}\\
			x'\ar[rr, tail] \ar{dd}\ar[rd, equal] &&w\drar\ar[two heads]{dd}\\
			&x'\ar[crossing over, tail]{rr}\ar[crossing over, leftarrow]{uu}&&y'\ar[two heads]{dd}\\
			0 \ar{rr} \ar[rd, equal] && z' \ar[rd, Isom] \\
			&0 \ar{rr}\ar[crossing over, leftarrow]{uu} && z \eqdot
			\end{tikzcd}
	\end{equation*}
	Here the front face is the original diagram, and the backsides are obtained by
	pushing out twice (in particular, $w = x' \amalg_x y$).
	As both the other front and the outer back faces are cocartesian, the morphism
	$z' \rightarrow z$ is an equivalence \cite[Rem.~6.2.4]{cisinski_higher_2019}. 
	Thus in the lower cube, the bottom, the back and the front face
	are bicartesian, hence the top face of the bottom cube fulfills the requirements of (\ref{enum:extend}). If the claim holds for
	that diagram, then the morphism $w \rightarrow y'$ is an equivalence, finishing the proof.

	So we assume, without loss of generality, that $f$ is the identity.
	We then extend the original diagram as indicated below. Note that all the new squares are bicartesian.
	\begingroup
	\setlength\arraycolsep{1.5pt}
		\begin{equation}\label{eq:huge_diagram_lemma_diagram}
		\begin{tikzcd}[column sep=6em]
			x \bicart \ar[d, equal], \ar[r, "\begin{pmatrix}
					\id \\
					-\id
				\end{pmatrix}"', near start] & x \oplus x \bicart[0.37em] \ar[d, "\begin{pmatrix}
					\id & 0\\
					\id & \id
				\end{pmatrix}"', Isom] \ar[r, "\begin{pmatrix}
					\id & 0\\
					0 & i
				\end{pmatrix}"', near start] & x \oplus y \ar[d, "\begin{pmatrix}
					\id & 0\\
					i & \id
				\end{pmatrix}"', Isom] \\[2em]
			x \bicart \ar[d] \ar[r, "\begin{pmatrix}
					\id \\
					0
				\end{pmatrix}"', near start] & x \oplus x \bicart \ar[r, "\begin{pmatrix}
					\id & 0\\
					0 & i
				\end{pmatrix}"', near start] \ar[d, "\begin{pmatrix}
					0 & \id
				\end{pmatrix}"'] & x \oplus y \ar[d, "\begin{pmatrix}
					0 & \id
				\end{pmatrix}"'] \\[0.9em]
			0 \ar[d, equal] \ar[r] \bicart & x \ar[d, equal] \ar[r, , tail, "i"] & y \ar[d, "g"] \\
			0 \ar[r] & x \ar[r, tail, "j"] & y'
		\end{tikzcd} 
	\end{equation}
	\endgroup
	The rightmost vertical morphism $\begin{pmatrix}
			j & g
		\end{pmatrix}\colon x \oplus y \rightarrow y'$ is a fibration, since it constitutes the leftmost
	vertical morphism in the diagram
	\begingroup
	\setlength\arraycolsep{1.5pt}
	\begin{equation*}
		\begin{tikzcd}[column sep=3em]
			x \oplus y \bicart \ar[d, "\begin{pmatrix}
				\id & 0 \\
				0 & g
			\end{pmatrix}"'] \ar[r, "\begin{pmatrix}
				0 & \id
			\end{pmatrix}", two heads] & y \ar[d, "g"] \\[1.8em]
			x \oplus y' \ar[d, "\begin{pmatrix}
				j & \id
			\end{pmatrix}"', two heads] \ar[r, "\begin{pmatrix}
				0 & \id
			\end{pmatrix}", two heads] & y' \ar[d, two heads] \\[1.2em]
			y' \ar[r, two heads] &z
		\end{tikzcd}
	\end{equation*}
	\endgroup
	where the rightmost vertical morphism is a fibration by assumption, the top square is readily
	seen to be bicartesian and the bottom square is cartesian by \cite[Lemma~4.6]{barwick_exact_2015}.
	
	As all small squares in the diagram \eqref{eq:huge_diagram_lemma_diagram} are cartesian, so is the outer square.
	As the right morphism is a fibration, it is cocartesian as well. The Pasting Lemma
	guarantees that the bottom right square is cocartesian, finishing the proof.
\end{proof}
\begin{prop}\label{prop:extend_exact_diagram}
	Let $\E$ be an exact \cat.
	For every square
	\begin{equation*}
		\begin{tikzcd}
			y \arrow[d, "f"] \arrow[r, two heads, "p"] & z \arrow[d, "g"] \\
			y' \arrow[r, two heads, "p'"] &z' \eqcomma
		\end{tikzcd}
	\end{equation*}
	there exist two morphisms between exact sequences $a\colon \sigma \rightarrow \sigma''$,
	$b\colon \sigma '' \rightarrow \sigma'$ and a composition $c = b \circ a$ such that
	the restriction along $\Delta^2 \rightarrow \Delta^1 \times \Delta^1$
	yields a commutative diagram of the form
    \begin{equation}\label{eq:exact_seq_two_times_two}
       \begin{tikzcd}
            \sigma\arrow[d, "a"]   & x \bicart \arrow[r, tail] \arrow[d]  & y \arrow[r, two heads] \arrow[d] & z \arrow[d, equal] \\
            \sigma''\arrow[d, "b"] & x' \arrow[r, tail] \arrow[d, equal] & w \bicart \arrow[r, two heads] \arrow[d] & z \arrow[d]                      \\
            \sigma'  & x' \arrow[r, tail]           & y' \arrow[r, two heads]          & z'                               \eqcomma
            \end{tikzcd}
    \end{equation}
    where the marked squares are bicartesian and the right rectangle is the original square.
\end{prop}
\begin{proof}
	By taking the pullback of $p'$ along $g$, we obtain a factorization of the square
	\begin{equation*}
		\begin{tikzcd}
			y \arrow[r, two heads, "p"] \arrow[d] & z \arrow[d, equal] \\
			w \bicart \arrow[r, two heads,  "p''"] \arrow[d] & z \arrow[d] \\
			y' \arrow[r, two heads, "p'"] & z' \eqdot
		\end{tikzcd}
	\end{equation*}
	Every exact sequence in $\Fun(\Delta^1 \times \Delta ^1, \E)$ is an iterated Kan
	extension of the fibration it restricts to.
	Hence, the restriction functor along the right vertical arrow
	\begin{equation*}
	\Fun(\Delta^1 \times \Delta ^1, \E) \rightarrow \Fun(\Delta ^1, \E)
	\end{equation*}
	restricts, by \HTT{Prop.~4.3.2.15}, to a trivial Kan fibration from the full subcategory
	spanned by the exact sequences to the full subcategory spanned by the fibrations. Hence, this diagram can be
	extended to a diagram of exact sequences of the form
	\begin{equation*}
		\begin{tikzcd}
			x \bicart \arrow[r, tail] \ar[d] & y \arrow[r, two heads, "p"] \arrow[d] & z \arrow[d, equal] \\
			x' \arrow[r, tail] \ar[d] & w \bicart \arrow[r, two heads,  "p''"] \arrow[d] & z \arrow[d] \\
			x'' \arrow[r, two heads, "p'"] & y' \arrow[r, two heads, "p'"] & z' \eqdot
		\end{tikzcd}
	\end{equation*}
	The top left square is bicartesian by \Cref{prop:big_diagram_lemma}. By the dual of \Cref{prop:big_diagram_lemma},
	the morphism $x' \rightarrow x''$ is an equivalence; identifying $x'$ with $x''$ along this equivalence finishes the proof.
\end{proof}
\begin{prop}\label{prop:diagram_lemma_1}
	Let $\E$ be an exact \cat, $p_1\colon y_1 \xrightarrowdbl{} z_1$ and $p_2\colon y_2 \xrightarrowdbl{} z_2$ two fibrations, and $g\colon y_1 \rightarrow z_2$
	a morphism in $\E$. The morphism $y_1 \oplus y_2 \rightarrow z_1 \oplus z_2$ given by the matrix 
	\begin{equation*}
	\begin{pmatrix} 		-p_1 & 0 \\ 		g & p_2 		\end{pmatrix}
	\end{equation*}
	is a fibration.
\end{prop}
\begin{proof}
	The map in question is a composition of the morphisms
	\begin{equation*}
		\begin{tikzcd}
		y_1 \oplus y_2 \arrow[rr, two heads, "\begin{pmatrix} 		\id & 0 \\ 		0 &p_2 		\end{pmatrix}"] & &
		y_1 \oplus z_2 \arrow[rr, two heads, "\begin{pmatrix} 		\id & 0 \\ 		g & \id 		\end{pmatrix}"] & &
		y_1 \oplus z_2 \arrow[rr, two heads, "\begin{pmatrix} 		-p_1 & 0 \\ 		0 & \id 		\end{pmatrix}"] & & 
		z_1 \oplus z_2
		\end{tikzcd}
	\end{equation*}
	where each map is a fibration: the two outer maps are pullbacks of $p_2$ and $-p_1$ along projections, and the one in the middle
	is an equivalence.
\end{proof}
\begin{acknowledgements}
This article originated from the author's master's thesis at the University of Bonn.
The author thanks his master's thesis supervisor Gustavo Jasso for suggesting
this topic, valuable and thorough feedback, fruitful discussions and constant
support in preparing this article, and for sharing his insight into 
proof strategies and obstructions in exact \cats.
\end{acknowledgements}
\emergencystretch=3em

\bibliographystyle{amsalpha}
\bibliography{literature}
\end{document}